\documentclass{article}
\usepackage[margin=1.2in]{geometry}

\usepackage[utf8]{inputenc}
\usepackage{graphicx}

\usepackage{amssymb}
\usepackage{amsmath}
\usepackage{amsthm}

\usepackage{caption}
\usepackage{subcaption}
\usepackage{algorithm}
\usepackage{algorithmic}
\usepackage{tikz}
\usepackage{dsfont}
\newtheorem{remark}{Remark}
\newtheorem{example}{Example}
\newtheorem{proposition}{Proposition}
\newtheorem{theorem}{Theorem}
\newtheorem{problem}{Problem}
\newtheorem{definition}{Definition}

\newcommand{\mat}[2]{\left [ \begin{array}{#1} #2 \end{array}\right ]}
\newcommand{\Hr}{\hat{H}}
\newcommand{\Is}{\mathcal{I}^n}
\newcommand{\Irs}{\mathcal{I}_r^n}
\newcommand{\Irso}{\mathcal{\hat{I}}}

\newcommand{\Es}{\mathcal{E}}

\usepackage{lipsum}
\usepackage{cleveref}
\makeatletter
\def\blfootnote{\xdef\@thefnmark{}\@footnotetext}
\makeatother

\crefname{section}{section}{sections}
\crefname{subsection}{subsection}{subsections}
\Crefname{section}{Section}{Sections}
\Crefname{subsection}{Subsection}{Subsections}
\crefname{problem}{problem}{problems}
\Crefname{figure}{Figure}{Figures}

\begin{document}

\title{Optimal Modal Truncation}
\author{Pierre Vuillemin$^\dagger$, Adrien Maillard$^\ddagger$ and Charles Poussot-Vassal$^\dagger$\\
{\small
$^\dagger$ONERA / DTIS, Universit\'e de Toulouse, F-31055 Toulouse, France
}\\
{\small \texttt{pierre.vuillemin@onera.fr}, \texttt{charles.poussot-vassal@onera.fr}
}\\
{
\small $^\ddagger$ Jet Propulsion Laboratory, California Institute of Technology, Pasadena, CA 91109, USA
}\\
{
\small \texttt{adrien.maillard@jpl.nasa.gov}
}
}

\date{}

\maketitle

\begin{abstract}
This paper revisits the modal truncation from an optimisation point of view. In particular, the concept of dominant poles is formulated with respect to different systems norms as the solution of the associated optimal modal truncation problem. The latter is reformulated as an equivalent convex integer or mixed-integer program. Numerical examples highlight the concept and optimisation approach.

\end{abstract}

\begin{slshape}
\end{slshape}

%\begin{AMS}
% 	93A15, 93A30, 93B40, 90C10, 90C11
  %68Q25, 68R10, 68U05
%\end{AMS}
\blfootnote{Submitted. Preprint version September 2020.}
\blfootnote{\textbf{Funding. }Part of this research was carried out at the Jet Propulsion Laboratory, California Institute of Technology, under a contract with the National Aeronautics and Space Administration (80NM0018D0004).}

\section{Introduction}

Large-scale dynamical models often arise in the industry due to the inherent complexity of the systems or phenomena to be studied and the complexity induced by the processes and tools  used for their modelling (e.g. Finite Element Methods, etc.). The dimension of these dynamical models then translates into a high numerical and computational burden that can prevent from performing simulation, analysis, control or optimisation. Model approximation is meant to alleviate the issue by building a much smaller model catching the main dynamics of the initial one and that could be used instead. In this context, this article is aimed at improving a standard linear model approximation method, the modal truncation, by adding considerations based on usual systems norms minimisation.

Let us consider consider a Linear Time-Invariant (LTI) dynamical model $H$ represented by its state-space realisation,
\begin{equation}
    \begin{array}{rcl}
    \dot{x}(t)&=& A x(t) + B u (t) \\
    y(t)&=& C x(t) + D u(t)
    \end{array}
    \label{eq:H}
\end{equation}
where $u(t) \in \mathbb{R}^{n_u}$ is the control input, $x(t) \in \mathbb{R}^n$ the internal state, $y(t) \in \mathbb{R}^{n_y}$ the output and $A$, $B$, $C$, $D$ are real matrices of adequate dimensions. Generally speaking, the objective of model approximation consists in finding a LTI model $\hat{H}$ described by 
\begin{equation}
    \begin{array}{rcl}
    \dot{\hat{x}}(t)&=& \hat{A} \hat{x}(t) + \hat{B} u (t) \\
    \hat{y}(t)&=& \hat{C} \hat{x}(t) + \hat{D} u(t)
    \end{array}
    \label{eq:Hr}
\end{equation}
where the number of input and outputs remains unchanged but the dimension of the state is decreased, i.e. $\hat{x}(t) \in \mathbb{R}^r$ with $r \ll n $, and such that the input to output behaviour of $\hat{H}$ is close to the one of $H$ in some sense.

For sake of simplicity the state-space representations \eqref{eq:H} and \eqref{eq:Hr} of $H$ and $\hat{H}$ are used indistinctly from their transfer functions, defined as follows,
\begin{equation}
    \begin{array}{rrcl}
    H:& \mathbb{C}\setminus \rho(A)&\to&\mathbb{C}^{n_y \times n_u}\\
    & s &\to&H(s) = C (s I - A)^{-1} B + D
    \end{array}
\end{equation}
where $\rho(A) = \{ \lambda \in \mathbb{C} / det(\lambda I - A) = 0 \}$ is the set that contains the eigenvalues of $A$.

Linear model approximation has been widely studied and several methods are now available. See e.g. \cite{antoulas:2005:approximation} for an overview of classical approaches and \cite{benner2017model,antoulas:2020:interpolatory} for a more recent treatment of the topic. Among the classical approaches, the modal truncation consists in projecting the large-scale model $H$ onto its dominant eigenspace. While the method is not the most efficient in general, it remains widely used in practice by engineers due to its conceptual simplicity and the fact that it preserves some poles of the initial model which are quantities of physical interest. Besides, it may be efficient enough to produce faithful reduced-order models, especially with poorly damped systems such as flexible structures.

In this context, the objective of this paper is to revisit the modal truncation from an optimisation point of view. In particular, usual systems norms are used to define specific dominant poles so that the resulting reduced-model is the best among all the possible models obtained by modal truncation with respect to the chosen norm. It is shown how this problem translates naturally into (binary) Integer Programming (IP) or Mixed Integer Programming (MIP). %The approach is then refined by allowing the residues of the reduced-order model to be adjusted during the process thus leading to Mixed Integer Programming (MIP). 
As the resulting problems share some similarities, various norms are considered: the $\mathcal{H}_2$ norm, its time and frequency-limited versions and the $\mathcal{H}_\infty$ norm.

The mandatory concepts and tools in systems theory and approximation are recalled in \cref{sec:prelim}, especially elements concerning the norms of systems that are considered in the article and about the modal approximation algorithm. Then in \cref{sec:contrib}, the latter is reformulated into different optimisation problems depending on the considered norm. Numerical applications are presented in \cref{sec:app} to highlight the concepts of the approach. Finally, concluding remarks are exposed in \cref{sec:ccl} together with some insights into possible extensions of this work.

Note that throughout this article, the following hypothesis are considered: 
\begin{itemize}
\item The model $H$ is assumed to be stable, i.e. $\rho(A)$ is contained within the open left complex plane $\mathbb{C}_-$. Indeed, should it have unstable poles, then they should be kept in the reduced-order model $\hat{H}$ anyway to have a bounded input-output approximation error.
\item The standard modal truncation approach requires the full eigenvalue decomposition of the matrix $A$ which involves $\mathcal{O}(n^3)$ dense linear algebra operations. Therefore the dimension $n$ of the state should remain moderate in practice. For iterative dominant eigenspaces computation, see \cite{Rommes:2006:subspaceAcc} and references therein.
\item Similarly, one assumes that $A$ has only semi-simple eigenvalues. While the approach could theoretically be extended in presence of Jordan blocks, their computation is ill-conditioned and could hardly be achieved in practice for non-trivial cases.
\end{itemize}
It should also be noted that a matrix $E \in \mathbb{R}^{n\times n}$ multiplying $\dot{x}$ in \eqref{eq:H} could be considered without affecting the reminder of the article. Indeed, should the matrix $E$ be non-singular, then it could be inverted to fall back on a system of the form \eqref{eq:H}. Otherwise, the associated transfer function would contain a polynomial part which elements of order larger or equal to $1$ should be kept for the truncation error to be finite, similarly to the unstable case.

\paragraph{Notations}Let us denote by $\mathbb{R}$ the set of real numbers, $\mathbb{R}_+$ the subset of positive numbers, $\mathbb{C}$ the set of complex numbers, $\mathbb{S}_{++}^n$ the set of symmetric and positive definite matrices of size $n$, $\mathcal{L}_2(I)$ the Lebesgues space of square integrable functions on $I$ ($\mathbb{R}$ if not specified). Given a complex valued matrix $M$, $M^T$ denotes its transpose, $M^*$ its conjugate and $M^H$ its conjugate transpose, $\Vert M \Vert_F$ is its Frobenius norm, $[M]_{i,j}$ is its $i,j$-th element, $tr(M)$ is its trace, $vec(M)$ is the vector obtained by concatenating the columns of $M$. Considering $z \in \mathbb{C}$, $Re(z)$ and $Im(z)$ are its real and imaginary parts, respectively, $\imath = \sqrt{-1}$ is the imaginary unit. Floor and ceiling functions are denoted by $\lfloor \cdot \rfloor$ and $\lceil \cdot \rceil$, respectively.

\section{Preliminaries}
\label{sec:prelim}

\subsection{Diagonal canonical form}
This Diagonal Canonical Form (DCF), given in \cref{prop:DCF}, appears naturally throughout the modal approximation and allows to simplify the expression of usual systems norms as recalled in the next section. See \cref{rq:H_practical_decomp} for its practical computation.

\begin{proposition}[Diagonal Canonical Form]
\label{prop:DCF}
Let us denote by $\Is= \{i=1,\ldots,n\}$ the set of integers ranging from $1$ to $n$. Provided that the dynamic matrix $A$ of $H$ in \eqref{eq:H} has semi-simple eigenvalues $\lambda_i \in \mathbb{C}$ ($i \in \Is$), then the transfer function associated with the system can be decomposed as follows,
\begin{equation}
    H(s) = D + \sum_{i\in \Is} \frac{\Phi_i}{s - \lambda_i},
    \label{eq:H_decomp}
\end{equation}
where $\Phi_i \in \mathbb{C}^{n_y \times n_u}$ is the residue associated with the pole $\lambda_i$.
\end{proposition}

\begin{remark}%[Practical computation of \eqref{eq:H_decomp}]
\label{rq:H_practical_decomp}
From a practical point of view, the decomposition \eqref{eq:H_decomp} is obtained by computing the full eigenvalue decomposition,
\begin{equation}
    A X = X \Delta 
\end{equation}
where $\Delta = diag(\lambda_1,\ldots,\lambda_n)$ and $X\in \mathbb{C}^{n\times n}$ contains the right eigenvectors of $A$. Then by applying the change of variable $ x(t) = X \xi(t)$, the dynamic \eqref{eq:H} becomes,
\begin{equation}
    \small
    \dot{\xi}(t) = \underbrace{X^{-1} A X}_{\Delta} \xi(t) + \underbrace{X^{-1} B}_{B_\Delta} u(t), \quad y(t) = \underbrace{CX}_{C_\Delta} \xi(t) + Du(t),
\end{equation}
and the residues are thus given, for $i\in \Is$, as
\begin{equation}
    \Phi_i = C_\Delta e_i e_i^T  B_\Delta = c_i b_i^T.
\end{equation}
Note that the residues are written as the outer product of two vectors $c_i \in \mathbb{C}^{n_y}$ and $b_i \in \mathbb{C}^{n_u}$ corresponding to the columns (resp. rows) of $C_\Delta$ (resp. $B_\Delta$) thus ensuring that $rank(\Phi_i) = 1$ and that the right hand side of \eqref{eq:H_decomp} is indeed of order $n$.
\end{remark}

\subsection{Systems norms}
This section is aimed at recalling some key elements concerning norms of LTI systems that are useful for this article. For a more exhaustive introduction to the topic, interested readers may refer to \cite[chap.4]{zhou:1995:robust} and references therein.

In particular, the $\mathcal{H}_2$-norm, its frequency and time limited counterparts are given in \cref{def:h2}, \cref{def:h2w} and \cref{def:h2t}, respectively. Provided the considered model is in DCF as in equation \eqref{eq:H_decomp}, these norms then have a simplified expressions as detailed in \cref{prop:h2pr}, \cref{prop:h2wpr} and \cref{prop:h2tpr}. In addition, the definition of the $\mathcal{H}_\infty$-norm is recalled in \cref{def:hi}.

\begin{definition}[$\mathcal{H}_2$-norm]\label{def:h2}
Considering a stable and strictly proper LTI model $H$ as in equation \eqref{eq:H}, its $\mathcal{H}_2$-norm is defined in the frequency domain as,
\begin{equation}
    \Vert H \Vert_2 \triangleq 
    \sqrt{
    \frac{1}{2\pi} \int_{-\infty}^{\infty} \Vert H(\imath \nu) \Vert_F^2 d\nu 
    }.
    \label{eq:h2def}
\end{equation}
\end{definition}

%
%The $\mathcal{H}_2$ norm is usually computed through the reachability or controllability gramians but the formulation based on the poles and residues of $H$ is more adequate for this article. In particular, 
\begin{proposition}
\label{prop:h2pr}
Assume that the stable and strictly proper LTI model $H$ is in DCF as in equation \eqref{eq:H_decomp}, then its $\mathcal{H}_2$-norm can be computed as follows,
\begin{equation}
\begin{array}{rcl}
 \displaystyle \Vert H \Vert_2^2 &=& \displaystyle \sum_{i\in \Is} tr(\Phi_i H(-\lambda_i)^T)\\
&=&   \displaystyle \sum_{i,k\in \Is}  \frac{tr(\Phi_i \Phi_k^T)}{-\lambda_i - \lambda_k} .
\end{array}
\label{eq:prh2}
\end{equation}
\end{proposition}

The $\mathcal{H}_2$-norm can be related to the time-domain in various ways. In particular, in the context of model reduction, let us consider the approximation error model $E = H - \Hr$, then for any input signal of bounded energy $u \in \mathcal{L}_2(\mathbb{R})$, the worst-case output error between the two models $\Vert y - \hat{y} \Vert_\infty$ is upper bounded as follows
\begin{equation}
    \Vert y - \hat{y} \Vert_\infty \leq \Vert H - \Hr \Vert_2 \Vert u \Vert_2.
    \label{eq:yh2bnd}
\end{equation}%

\begin{definition}[Frequency-limited $\mathcal{H}_2$-norm]
\label{def:h2w}
The frequency-limited $\mathcal{H}_2$-norm \cite{gawronski:2004:advanced}, denoted $\mathcal{H}_{2,\omega}$-norm, is defined by restricting the frequency interval to $[-\omega, \omega]$ in \eqref{eq:h2def} so that 
\begin{equation}
    \Vert H \Vert_{2,\omega} \triangleq 
    \sqrt{
    \frac{1}{2\pi} \int_{-\omega}^{\omega} \Vert H(\imath \nu) \Vert_F^2 d\nu 
    }.
    \label{eq:h2wdef}
\end{equation}
\end{definition}
%\paragraph{$\mathcal{H}_{2,\omega}$-norm. }By restricting the frequency interval of the integral to $[-\omega, \omega]$ in \eqref{eq:h2def}, one can define the frequency-limited $\mathcal{H}_2$-norm, denoted $\mathcal{H}_{2,\omega}$-norm \cite{gawronski:2004:advanced}. Unlike the $\mathcal{H}_2$-norm, it remains finite even when $D \neq 0$ and it can be computed similarly to \eqref{eq:prh2} as follows \cite{vuillemin:2014:poles},
\begin{proposition}
\label{prop:h2wpr}
The frequency-limited $\mathcal{H}_2$-norm of a LTI model $H$ in DCF form can be be computed similarly to \eqref{eq:prh2} as follows,
\begin{equation}
\small 
\Vert H \Vert_{2,\omega}^2 = \frac{\omega}{\pi} \Vert D\Vert_F^2 - \frac{2}{\pi} \sum_{i\in \Is} tr(\Phi_i H(-\lambda_i)^T) atan \left (\frac{\omega}{\lambda_i} \right ),
    \label{eq:prh2w}
\end{equation}
where $atan(z) = \frac{1}{2j} (log(1+\imath z) - log(1-\imath z))$ and $log(z)$ is the principal value of the complex logarithm for $z\neq 0$.
\end{proposition}
\begin{proof}
See \cite{vuillemin:2014:poles}.
\end{proof}

Note that $\mathcal{H}_{2,\omega}$ is only a semi-norm when considering the whole Lebesgue space $\mathcal{L}_2(\imath \mathbb{R})$ but it is a norm for rational functions as considered here.

Let $h(t)$ denotes the impulse response of $H$, corresponding to the inverse Laplace transform of $H(s)$, i.e.
\begin{equation}
    h(t) \triangleq \mathcal{L}^{-1}(H)(t) = Ce^{At}B + D\delta(t),
\end{equation}
where $\delta$ is the Dirac impulse. Due to Parseval's equality, for a stable and strictly proper model $H$, the $\mathcal{H}_2$-norm can also be computed in time-domain as follows,
\begin{equation}
    \Vert H \Vert_{2}=\Vert h \Vert_{2} \triangleq  \sqrt{\int_{-\infty}^{\infty} \Vert h(t) \Vert_F^2 dt}.
    \label{eq:h2time}
\end{equation}
For a stable system $H$, $h(t) = 0$ for $t<0$, therefore the integral in \eqref{eq:h2time} can be restricted to $\mathbb{R}_+$. In addition, by restricting even further the integration interval to $[0, \tau]$, one can define the time-limited $\mathcal{H}_2$-norm \cite{gawronski:2004:advanced,goyal:2019:timelim} as detailed in \cref{def:h2t}.

\begin{definition}[$h_{2,\tau}$-norm]
\label{def:h2t}
Considering a stable and strictly proper LTI model $H$ as in \eqref{eq:H} with impulse response $h(t)$, its time-limited $\mathcal{H}_2$-norm, denoted $h_{2,\tau}$-norm here, is defined for $\tau >0$ as
\begin{equation}
    \Vert h \Vert_{2,\tau}^2 \triangleq  \int_{0}^{\tau} \Vert h(t) \Vert_F^2 dt.
\end{equation}
\end{definition}

\begin{proposition}
\label{prop:h2tpr}
The time-limited $h_2$-norm of a LTI model $H$ in DCF form can be computed similarly to \eqref{eq:prh2} as follows,
\begin{equation}
    \Vert h \Vert_{2,\tau}^2 = \sum_{i,k\in \Is} \frac{tr(\Phi_i\Phi_k^T)}{\lambda_i +  \lambda_k } \left ( e^{(\lambda_i + \lambda_k) \tau} -1 \right ).
    \label{eq:h2tpr}
\end{equation}
\end{proposition}
\begin{proof}
The DCF \eqref{eq:H_decomp} with $D = 0$ enables to re-write the impulse response as
\begin{equation}
    h(t) = \sum_{i\in \Is} \Phi_i e^{\lambda_i t},
\end{equation}
which naturally leads to the expression \eqref{eq:h2tpr} after integration.
\end{proof}

Again, $h_{2,\tau}$ is only a semi-norm for the whole space of square integrable functions $\mathcal{L}_2(\mathbb{R})$, but it is a norm for the impulse response functions associated with rational functions.

A time-domain bound of the error similar to the one available with the $\mathcal{H}_2$-norm \eqref{eq:yh2bnd} can be derived \cite{goyal:2019:timelim}.

\begin{definition}[$\mathcal{H}_\infty$-norm]
\label{def:hi}
Considering a stable LTI model $H$ as in \eqref{eq:H}, its $\mathcal{H}_\infty$-norm is defined as follows,
\begin{equation}
    \Vert H \Vert_\infty \triangleq  \sup_{\nu \in \mathbb{R}} \sigma_1(H(\imath \nu)),
\end{equation}
where $\sigma_1(H(\imath \nu))$ is the largest singular value of the transfer matrix. 
\end{definition}
The $\mathcal{H}_\infty$-norm represents the worst amplification gain of the system and is a widely used measure of robustness. Similarly to the $\mathcal{H}_2$-norm, within the context of model reduction, it enables to bound the $\mathcal{L}_2$ gain of the approximation error,
\begin{equation}
    \Vert y - \hat{y} \Vert_2 \leq \Vert H - \Hr \Vert_\infty \Vert u \Vert_2.
\end{equation}
Note that the computation of the $\mathcal{H}_\infty$-norm either requires an iterative bisection procedure or the resolution of a Semi Definite Program (SDP) \cite{Scherer:1997}.

\subsection{Reminder on the modal truncation}
Modal truncation consists in keeping only $r$ elements from the decomposition \eqref{eq:H_decomp} to form the reduced-order model. In particular, by defining the subset $\Irs \subset \Is$ containing $r$ unique elements from $\Is$,
\begin{equation}
    \hat{H}(s) = \hat{D} + \sum_{i=1}^r \frac{\hat{\Phi}_i}{s - \hat{\lambda}_i} = \hat{D} + \sum_{i\in \Irs} \frac{\Phi_i}{s - \lambda_i}.
    \label{eq:HIrs}
\end{equation}
Note that for $\Hr$ to have a real realisation, the retained complex eigenvalues of $H$ must be selected together with their complex-conjugate pair. This implies that some combinations are not allowed to form $\Irs$. %: a model with only complex poles cannot be reduced to a model with an odd order $r$. One assumes that this is satisfied by the set $\Irs$. 
In standard modal truncation, $\hat{D}$ is chosen equal to $D$.

Modal truncation then boils down to select the \emph{dominant} poles-residues couples that should be included within $\Irs$. Dominant poles may be defined in various ways. Below, the definition based on the bound of the $\mathcal{H}_\infty$-norm of the approximation error that is generally considered in the reduction literature is recalled. %In section \ref{sec:contrib}, an alternative definition based on optimisation consideration

%alternative dominant poles may be defined based on optimality considerations.

\paragraph{Approximation error and bounds}Let us denote by $\Es = \Is \setminus \Irs$ the set of discarded indexes. The approximation error $E$ between $H$ and $\Hr$ is then naturally given as,
\begin{equation}
    E(s) = H(s) - \Hr(s) = D - \hat{D} + \sum_{i\in \Es} \frac{\Phi_i}{s - \lambda_i}.
\end{equation}

For its $\mathcal{H}_2$ norm to be bounded, its direct feedthrough must be zero, i.e. $\hat{D}- D$ must be zero.  In that case, $\Vert E \Vert_2$ is readily obtained by considering the specific formulation of the norm for transfer functions with such structure \eqref{eq:prh2}, 
\begin{equation}
    \Vert E \Vert_2^2 =  \sum_{i\in \Es} tr(\Phi_i E(-\lambda_i)^T).
    \label{eq:h2err}
\end{equation}

For the $\mathcal{H}_\infty$-norm, only an upper bound has been derived based on the triangular inequality. For $D = \hat{D}$, it states that,
\begin{equation}
    \Vert E \Vert_\infty \leq \sum_{i\in\Es} \frac{\Vert \Phi_i \Vert_2}{|Re(\lambda_i)|}.
    \label{eq:hibound}
\end{equation}

\paragraph{Usual criterion for determining $\Irs$}Dominant poles are generally defined as the poles $\lambda_i$ that have the largest ratio
\begin{equation}
    \frac{\Vert \Phi_i\Vert_2}{|Re(\lambda_i)|}.
    \label{eq:hi_crit_poles}
\end{equation}
Such a choice to fill the set $\Irs$ enables to minimise the $\mathcal{H}_\infty$ bound \eqref{eq:hibound} of the approximation error. Still, it does not make the resulting reduced-order model optimal with respect to the $\mathcal{H}_\infty$-norm and we shall see in \cref{sec:contrib} that optimality considerations allow to characterise dominant poles in a more generic way.

%\subsection{Some considerations of optimal model approximation}

\section{Optimal modal truncation}
\label{sec:contrib}

The main idea here consists in formulating the modal truncation method as an optimisation problem. The latter is stated formally in \cref{pb:init}. Dominant poles-residues are then defined in \cref{def:dpr} as the elements associated to the corresponding optimal solution.

\begin{problem}[Optimal modal truncation]
\label{pb:init}
Let us consider a stable $n$-th order LTI dynamical model $H$ in DCF \eqref{eq:H_decomp}, a reduction order $0<r<n$ and the set $\Irs$ of indexes containing the elements to be kept within reduced-order model $\Hr$ as in \eqref{eq:HIrs}. Considering in addition some system norm $\Vert \cdot \Vert$ (e.g. $\mathcal{H}_2$, $\mathcal{H}_\infty$, etc.), the optimal modal truncation problem can then be formally stated as
\begin{equation}
\begin{array}{rl}
    \displaystyle \min_{\Irs}& \displaystyle \Vert H - \Hr \Vert\\
    s.t.&\\
    &\Hr\text{ given by \eqref{eq:HIrs}}\\
    &\Hr\text{ has real coefficients}
\end{array}
\label{eq:pb:init}
\end{equation}
\end{problem}

\begin{definition}[Dominant poles-residues]
\label{def:dpr}
Suppose that $\Irso_r^n$ solves the optimal modal approximation problem \eqref{eq:pb:init}, then the set of $r$ dominant poles-residues with respect to the associated system norm is defined as
\begin{equation}
    \Lambda(\Irso_r^n) =  \{\lambda_i, \Phi_i\}_{i\in \Irso_r^n}
\end{equation}

\end{definition}

As illustrated in \cref{ex:toy}, solving \cref{pb:init} boils down to select the $r$ poles-residues amongst $n$ that minimise the error. Note that as complex poles must come by pairs, the exact number of possible unique combinations depends on their number within the initial model $H$ as detailed in \cref{maillardConjecture}.

\begin{proposition}
\label{maillardConjecture}
Considering a $n$-th order LTI model $H$ with $n_c$ pairs of complex poles and $n_r = n  - 2 n_c$ real poles, then the number of possible combinations for $r$-th order modal truncation ($r<n$) with a real state-space realisation is 

\begin{equation}
    \kappa(n_r, n_c, r) =     \begin{cases}
          \binom{n_r}{r} & \text{if } n_c=0\\
         \binom{n_c}{N_e} &\text{if } n_r = 0\\
     \sum_{i=N_s}^{N_e} \binom{n_c}{i} \binom{n_r}{r-2i} & \text{if } n_r >0, n_c > 0
    \end{cases}
\end{equation}
with $N_s = \max(0,\lceil \frac{r-n_r}{2} \rceil)$ and $N_e = \min(n_c,\lfloor \frac{r}{2} \rfloor)$.
\end{proposition}
\begin{proof}

Case 1 $(n_c = 0)$ : if there is no complex pole in the model, $n = n_r$, the binomial coefficient counts how many sets of size $r$ formed of $n_r$ real poles exist. 
Case 2 $(n_r = 0)$ : if there is no real pole in the model, $n = n_c$ and we have to count as previously. As the set of poles in an r-modal truncation is of size $r$ and the complex poles come by pairs, the quotient $N$ resulting from the division of $r$ by $2$ represents the maximum number of complex poles in an r-modal truncation. Case 3  $(n_c >0, n_r>0)$ : in between $N_s$ and $N_e$, the formula uses the same principle and sums over the different configurations that the set of poles can take when there is a combination of real and complex poles. $N_s$ and $N_e$ are set so binomial coefficients are always defined. When $i = N_s$, there is no or the minimum number (depending the number of real poles) of complex pole in the truncation. When $i=N_e$ and $r$ is even,  the set of poles is only made of complex poles ($r - 2i = 0$) and the first term is equal to the second case. When $i=N_e$ and $r$ is odd, the set contains only $N_e$ complex poles and one real pole. The right term equals $\binom{n_r}{1} = n_r$ which effectively counts the number of possibilities to fill the single slots in the sets of complex poles. 
\end{proof}

While the number of combinations $\kappa$ grows slower than $\binom{n}{r}$ as soon as the model $H$ has some complex eigenvalues, it is still large enough to prevent from scanning exhaustively the decision tree.

\begin{example}
\label{ex:toy}
Let us consider the following third order model 
\begin{equation}
    \small
    H(s) = H_1(s) + H_2(s) + H_3(s) = \frac{1}{s + 1} + \frac{1}{s + 3} + \frac{2}{s + 5},
    \label{eq:toy}
\end{equation}
that should be reduced to $r=2$. In this simple case and considering e.g. the $\mathcal{H}_2$-norm, one can evaluate all the possible combinations. \cref{fig:toy} shows a tree where each node represents the subsystem which is added in the reduced-order model. For instance, the leftmost branch leads to $H_1 + H_2$. The approximation error is also displayed next to each node. Note that this tree contains redundant branches (in grey) due to commutativity of the sum. As only real poles are considered, at the end, there is ${n\choose{r}} = \kappa(3,0,2)  = 3$ unique possible combinations and the optimal model $\Hr$ is clearly given by
\begin{equation}
    \Hr(s) = H_1(s) + H_3(s),
\end{equation}
meaning that $\{-1,1\}$ and $\{-5,2\}$ are the dominant poles-residues for the $\mathcal{H}_2$-norm.

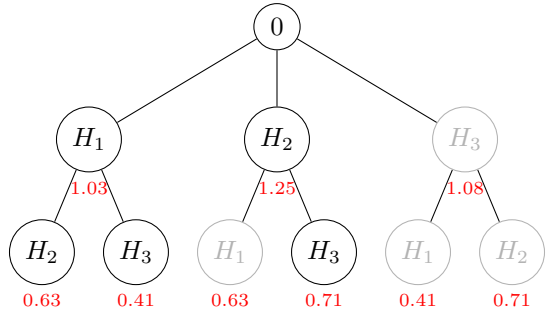
\begin{figure}
    \centering
    \begin{tikzpicture}[level/.style={sibling distance=25mm/#1}]
    \node [circle,draw] (z){$0$}{
      child {node [circle,draw] (a) {$H_1$}
        child {node [circle,draw] (a2) {$H_2$}}
        child {node [circle,draw] (a3) {$H_3$}}
      }
      child {node [circle,draw] (b) {$H_2$}
        child {node [circle,draw,color=black!30!white] (b1) {$H_1$}}
        child {node [circle,draw] (b3) {$H_3$}}
      }
      child {node [circle,draw,color=black!30!white] (c) {$H_3$}
        child {node [circle,draw,color=black!30!white] (c1) {$H_1$}}
        child {node [circle,draw,color=black!30!white] (c2) {$H_2$}}      
      }
    };
    % errors 
    \node at (a.south) [font=\scriptsize,color=red,anchor = north] {$1.03$};
    \node at (b.south) [font=\scriptsize,color=red,anchor = north] {$1.25$};
    \node at (c.south) [font=\scriptsize,color=red,anchor = north] {$1.08$};

    \node at (a2.south) [font=\scriptsize,color=red,anchor = north] {$0.63$};
    \node at (a3.south) [font=\scriptsize,color=red,anchor = north] {$0.41$};
    
    \node at (b1.south) [font=\scriptsize,color=red,anchor = north] {$0.63$};
    \node at (b3.south) [font=\scriptsize,color=red,anchor = north] {$0.71$};
    
    \node at (c1.south) [font=\scriptsize,color=red,anchor = north] {$0.41$};
    \node at (c2.south) [font=\scriptsize,color=red,anchor = north] {$0.71$};
    
    \end{tikzpicture}
    \caption{Tree of possible combinations for the second order optimal approximation of \eqref{eq:toy} with respect to the $\mathcal{H}_2$-norm.}
    \label{fig:toy}
\end{figure}

\end{example}

To reformulate \cref{pb:init} in a practical way, let us consider the following parametrization of the reduced-order model,
\begin{equation}
    H_\alpha(s) = D_\alpha + \sum_{i \in \Is} \alpha_i \frac{\Phi_i}{s - \lambda_i},
    \label{eq:Ha}
\end{equation}
where $\alpha_i = e_i^T \alpha\in \{0,1\}$ are binary variables acting as activation variables and $D_\alpha \in \mathbb{R}^{n_y \times n_u}$. With this parametrization, the order of the reduced model may be enforced by the constraint
\begin{equation}
    \mathds{1}^T \alpha = r,
    \label{eq:cst_order}
\end{equation}
where $\mathds{1}\in \mathbb{R}^n$ is a vector full of ones. In addition, to ensure that $H_\alpha$ has a real realisation, the complex conjugate pairs of poles must be kept together. This translates into additional linear constraints between the $\alpha_i$. Indeed, let us consider the set of $n_c$ complex conjugate pairs indexes,
\begin{equation}
    \mathcal{I}^C = \{ \{i,j\}\in \Is / \lambda_i \in \mathbb{C},\, Im(\lambda_i)>0,\, \lambda_i = \bar{\lambda}_j \},
    \label{eq:Icdef}
\end{equation}
and the matrix $M \in \mathbb{R}^{n_c\times n}$ which contains a row for each couple $\{i,j\} \in \mathcal{I}^C$ such that,
\begin{equation}
    [M]_{\cdot,i} =  - [M]_{\cdot,j} = 1. 
\end{equation}
Then the realness constraint is 
\begin{equation}
    M \alpha = 0.
    \label{eq:cst_realness}
\end{equation}
By coupling the parametrization \eqref{eq:Ha} with the constraints \eqref{eq:cst_order}, \eqref{eq:cst_realness} and the binary constraint, \cref{pb:init} may be reformulated as a binary optimisation problem as stated in \cref{pb:initbin}.

\begin{problem}[Binary formulation of optimal modal approximation]
\label{pb:initbin}
Considering the parametrization \eqref{eq:Ha} for the reduced-order model, \cref{pb:init} is equivalent to the following problem,
\begin{equation} 
\begin{array}{lc}
     \displaystyle \min_{\alpha}&   \Vert H - H_\alpha \Vert \\
     s.t.&
     \begin{array}[t]{rcl}
     \\
     \alpha&\in&\{ 0,1\}^n \\
     \mathds{1}^T \alpha&=&r\\
     M \alpha&=&0
     \end{array}
\end{array}
\label{eq:pb:initeq}
\end{equation}
\end{problem}

In the following section, \cref{pb:initbin} is specified for $\mathcal{H}_2$-norm, its frequency/time limited counterparts and the $\mathcal{H}_\infty$-norm.

\subsection{In $\mathcal{H}_2$-norm}
\label{ssec:h2opt}
Considering the framework introduced in \cref{pb:initbin}, the optimal  $\mathcal{H}_2$ modal truncation problem can be recasted as a convex binary quadratic problem as stated in \cref{thm:h2pb}.

\begin{theorem}[Optimal $\mathcal{H}_2$ modal truncation]
\label{thm:h2pb}
Considering the notations of \cref{pb:initbin}, the $r$-th order optimal $\mathcal{H}_2$ modal truncation model $H_\alpha^\star$ is such that $D_\alpha^\star = D $ and $\alpha^\star$ is the solution of the following convex quadratic binary problem,
\begin{equation} 
\begin{array}{lc}
     \displaystyle \min_{\alpha}&   (\mathds{1} - \alpha)^T Q (\mathds{1} - \alpha)\\
     s.t.&
     \begin{array}[t]{rcl}
     \\
     \alpha&\in&\{ 0,1\}^n \\
     \mathds{1}^T \alpha&=&r\\
     M \alpha&=&0
     \end{array}
\end{array}
\label{eq:pbh2}
\end{equation}
where $Q \in \mathbb{C}^{n\times n}$ is a hermitian matrix which entries are given as,
\begin{equation}
    [Q]_{i,j} = \displaystyle \frac{tr(\Phi_i \Phi_j^H)}{-\lambda_i - \lambda_j^*},\quad i,j = 1,\ldots,n
    \label{eq:Qmat}
\end{equation}
\end{theorem}
\begin{proof}
For the $\mathcal{H}_2$-norm of the approximation error $E_\alpha = H - H_\alpha$ to be finite, $D_\alpha$ is constrained to be equal to $D$ and may be discarded in the sequel. In that case, $E_\alpha$ is given as
\begin{equation}
    E_\alpha(s) =  H(s) - H_\alpha(s) = \sum_{i \in \Is} (1 - \alpha_i) \frac{\Phi_i}{s - \lambda_i},
    \label{eq:Ea}
\end{equation}
and from equation \eqref{eq:prh2}, the quadratic nature of the error w.r.t. $\alpha$ appears,
\begin{equation}
    \Vert E_\alpha \Vert_2^2 = \sum_{i,k \in \Is} (1 - \alpha_i) \frac{tr(\Phi_i \Phi_k^T)}{-\lambda_i - \lambda_k} (1 - \alpha_k).
    \label{eq:h2err_iqp}
\end{equation}
As each complex pole-residue pair in $H$ comes with its complex conjugate, the sums in the right-hand side of \eqref{eq:h2err_iqp} may be reordered so that $\{\lambda_k,\Phi_k^T\}$ is replaced by their conjugate. The $\mathcal{H}_2$ error can then be rewritten as
\begin{equation}
    \Vert E_\alpha \Vert_2^2 = (\mathds{1} - \alpha)^T Q (\mathds{1} - \alpha).
    \label{eq:h2err1}
\end{equation}
Coupling the objective function \eqref{eq:h2err1} with the constraints \eqref{eq:cst_order}, \eqref{eq:cst_realness} and the binary constraint leads to the optimisation problem \eqref{eq:pbh2}. 

Its objective is a squared-norm and is therefore strictly convex. The equality constraints of problem \eqref{eq:pbh2} are linear and thus convex. Therefore, it is a binary convex quadratic problem.
\end{proof}

As the relaxation of the binary problem \eqref{eq:pbh2} is convex, efficient branch and bounds algorithms (see e.g. \cite{bliek1u:2014:solving}) can be used to solve the overall optimal $\mathcal{H}_2$ truncation problem. 

Yet, state of the art solvers may not handle the fact that $Q$ is complex. However, as each complex element comes with its conjugate in the sum \eqref{eq:h2err1} (the overall sum is real), $Q$ can be replaced with $\tilde{Q} = \frac{1}{2} (Q + Q^H)$ which leads to the same objective function as long as $M \alpha = 0$.

\paragraph{About the initialisation}While existing general purpose solvers are perfectly able to determine a feasible starting point, providing a meaningful feasible initial solution may help to prune rapidly some parts of the tree.

As highlighted in \cref{prop:h2bound}, the $\mathcal{H}_2$-norm of the approximation error is upper bounded by the sum of each subsystems norms. Therefore, this suggests to initially select the pole-residue pairs with largest associated $\mathcal{H}_2$-norm \eqref{eq:h2crit_poles}. Again, complex conjugate pairs must be kept together meaning that corner cases have to be dealt with. In particular, if one pole remains to be selected but the next largest couple is complex, then either discard it until a real pole is reached or discard the last real pole selected to get the complex pair. Alternatively in those cases, decrease or increase $r$ by one. The initialisation process is highlighted in \cref{ex:toyinit}.

%Assume $\Irs$ denotes the set of selected indexes and $\Es = \Is \setminus \Irs$ is the set of discarded indexes. Then, by denoting $H_i = \frac{\Phi_i}{s - \lambda_i}$, $i \in \Es$, the approximation error is 
%\begin{equation}
%    E_\alpha = \sum_{i\in \Es} H_i.
%    \label{eq:err_decomp}
%\end{equation}
%Similarly to the derivation of $\mathcal{H}_\infty$ bound \eqref{eq:hibound}, by applying the triangular inequality,
%\begin{equation}
%    \Vert E_\alpha \Vert_2 \leq \sum_{i\in \Es} \Vert H_i \Vert_2.
%    \label{eq:h2bound}
%\end{equation}
%Therefore, the $\mathcal{H}_2$-norm of the error is bounded by the sum of the $\mathcal{H}_2$-norms of discarded subsystems. This completes the bound \eqref{eq:hibound} on the $\mathcal{H}_\infty$-norm.

\begin{proposition}
\label{prop:h2bound}
The $\mathcal{H}_2$-norm of the approximation error $E_\alpha$ is bounded by the sum of the discarded subsystems norms,
\begin{equation}
    \Vert E_\alpha \Vert_2 \leq \sum_{i\in \Es} \Vert H_i \Vert_2,
    \label{eq:h2bound}
\end{equation}
where the individuals subsystems norms are given as 
\begin{equation}
    \Vert H_i \Vert_2 = \frac{\Vert \Phi_i \Vert_F}{\sqrt{-2Re(\lambda_i)}}.
    \label{eq:h2crit_poles}
\end{equation}
\end{proposition}
\begin{proof}
Applying the triangular inequality to the norm of the approximation error $E_\alpha$ directly leads to the result.
\end{proof}

%Equation \eqref{eq:h2bound} suggests to select the pole-residue couples associated with the subsystem $H_i$ of largest $\mathcal{H}_2$-norm. The latter is readily given as
%\begin{equation}
%    \Vert H_i \Vert_2 = \frac{\Vert \Phi_i \Vert_F}{\sqrt{-2Re(\lambda_i)}}.
%    \label{eq:h2crit_poles}
%\end{equation}
%Again, complex conjugate pairs must be kept together meaning that corner cases have to be dealt with. In particular, if one pole remains to be selected but the next largest couple is complex, then either discard it until a real pole is reached or discard the last real pole selected to get the complex pair. Alternatively in those case, decrease or increase $r$ by one.

\begin{example}
\label{ex:toyinit}
Considering the trivial model of \cref{ex:toy}, the $\mathcal{H}_2$ criterion \eqref{eq:h2crit_poles} indicates,
\begin{equation}
\begin{array}{rcccl}
    \Vert H_1\Vert_2 &=& 1/\sqrt{2}&\approx &0.71,\\
    \Vert H_2\Vert_2 &=& 1/\sqrt{6}&\approx &0.41,\\
    \Vert H_3\Vert_2 &=& 2/\sqrt{10}&\approx & 0.63.
\end{array}
\end{equation}
This suggests $H_1 + H_3$ as reduced model, which is indeed the optimal one in that case. Note that the $\mathcal{H}_\infty$ criterion \eqref{eq:hi_crit_poles} leads to the same ordering here but this could be otherwise for multiple inputs multiple outputs cases.%\footnote{For $n_y = n_u = 1$, $\Vert \Phi_i \Vert_2 = \Vert \Phi_i \Vert_F = |\Phi_i|$ and the functions $1/x$ and $1/\sqrt{2x}$ intersects for $x>0$.}.
\end{example}

%Since the $\mathcal{H}_2$-norm is actually a inner product induced norm (see e.g. \cite[chap.4]{zhou:1995:robust}), the approximation error \eqref{eq:h2err} can be alternatively written as,
%\begin{equation}
%    \Vert E_\alpha \Vert_2^2 = \sum_{i\in\Es} \Vert H_i \Vert_2^2 + 2\sum_{\scriptsize
%    \begin{array}{c}i,k \in \Es\\i\neq k \end{array}} < H_i, H_k>
%    \label{eq:h2err_inner}
%\end{equation}
%with $H_i = \frac{\Phi_i}{s - \lambda_i}$, $i \in \Es$. Applying Cauchy-Schwartz inequality to \eqref{eq:h2err_inner} leads to 
%\begin{equation}
%    \Vert E_\alpha \Vert_2 \leq \sum_{i\in \Es} \Vert H_i \Vert_2^2.
%    \label{eq:h2bound}
%\end{equation}
%Therefore, the squared $\mathcal{H}_2$-norm of the error is bounded by the sum of the squared $\mathcal{H}_2$-norms of discarded subsystems. This completes the bound \eqref{eq:hibound} on the $\mathcal{H}_\infty$-norm.

%\paragraph{Alternate form. }

\subsection{In $\mathcal{H}_{2,\omega}$-norm}
\label{ssec:h2wopt}
Similarly to the $\mathcal{H}_2$-case, the optimal $\mathcal{H}_{2,\omega}$ modal truncation problem can be recasted as an optimisation problem as stated in Theorem \ref{thm:h2w}. The main difference lies in the fact that for $\omega < \infty$, the frequency-limited $\mathcal{H}_2$-norm remains finite even when there is a direct feedthrough. Therefore, unlike in the $\mathcal{H}_2$ case, $D_\alpha \in \mathbb{R}^{n_y \times n_u}$ remains a free design variable in the parametrisation \eqref{eq:Ha} and the resulting optimisation problem is thus a mixed convex quadratic program.

\begin{theorem}[Optimal $\mathcal{H}_{2,\omega}$ modal truncation]
\label{thm:h2w}
Considering the framework of \cref{pb:initbin}, let us define $\bar{\alpha} = 1 - \alpha\in \mathbb{R}^n$, $m = n + n_y n_u$, $d = vec(D) \in \mathbb{R}^{n_y n_u}$, $d_\alpha = vec(D_\alpha)\in \mathbb{R}^{n_y n_u}$, $a_\omega(\lambda_i) = atan(\frac{\omega}{\lambda_i})$ and $x_\alpha = [\bar{\alpha}^T,d_\alpha^T]^T \in \mathbb{R}^{m}$. Then, the $r$-th order optimal $\mathcal{H}_{2,\omega}$ modal truncation model $H_\alpha^\star$ is obtained through $x_\alpha^\star$, solution of the following strictly convex mixed quadratic problem,
\begin{equation} 
\begin{array}{lc}
     \displaystyle \min_{\alpha,d_\alpha}&   x_\alpha^T \mathcal{Q}_\omega x_\alpha + c_\omega^T x_\alpha %+ q_\omega
     \\
     s.t.&
     \begin{array}[t]{rcl}
     \\
     \alpha&\in&\{ 0,1\}^n \\
     d_\alpha&\in&\mathbb{R}^{n_y n_u}\\
     \mathds{1}^T \alpha&=&r\\
     M \alpha&=&0
     \end{array}
\end{array}
\label{eq:pbh2w}
\end{equation}
where $\mathcal{Q}_\omega \in \mathbb{C}^{m \times m}$ is a hermitian matrix defined as 
\begin{equation}
    \mathcal{Q}_\omega =  \mat{cc}{Q_\omega & \frac{1}{2} U_\omega\\
     \frac{1}{2} U_\omega^H& \frac{\omega}{\pi} I_{n_y n_u}}
\end{equation}
where the top left block is related to the matrix $Q$ \eqref{eq:Qmat}, for $i,j =1,\ldots,n$
\begin{equation}
    [Q_\omega]_{i,j} = -\frac{1}{\pi} \left (a_\omega(\lambda_i) + a_\omega(\lambda_k^*)\right )[Q]_{i,j}
    \label{eq:Qwmat}
\end{equation}
and the off-diagonal term is given, for $i = 1,\ldots,n$, as
\begin{equation}
    e_i^T U_{\omega} = -\frac{2}{\pi} a_\omega(\lambda_i) vec(\Phi_i)^T.
\end{equation}
Additionally, the linear term is given by
\begin{equation}
%\scriptsize
    c_\omega = -\frac{2}{\pi}\mat{ccccc}{ a_\omega(\lambda_1) vec(\Phi_1)^T&\cdots& a_\omega(\lambda_n) vec(\Phi_n)^T &\omega d^T}^T.
\end{equation}
%and the constant is $q_\omega =  \frac{\omega}{\pi} tr(D^T D)$.
\end{theorem}
\begin{proof}
Based on equation \eqref{eq:prh2w}, the norm of the approximation error can be written as three components,
\begin{equation}
%\scriptsize
\begin{array}{rl}
    \Vert E_\alpha \Vert_{2,\omega}^2 =& E_{1,\omega}+E_{2,\omega} + E_{3,\omega}\\
    =&\displaystyle \frac{\omega}{\pi} \Vert D_e \Vert_F^2 \ldots 
    \\
    & \displaystyle-\frac{2}{\pi}\sum_{i\in \Is} (1-\alpha_i) tr(\Phi_i D_e^T) atan(\frac{\omega}{\lambda_i}) \ldots\\ 
    & \displaystyle - \frac{2}{\pi} \sum_{i,k \in \Is} (1-\alpha_i) \frac{tr(\Phi_i \Phi_k^T) atan(\frac{\omega}{\lambda_i})}{-\lambda_i -\lambda_k}(1-\alpha_k),
    \end{array}
    \label{eq:h2werror1}
\end{equation}
where $D_e = D - D_\alpha$. 

Similarly to the $\mathcal{H}_2$ case, the error \eqref{eq:h2werror1} is also quadratic with respect to the optimisation variables. Indeed, first, note that the terms within the sums of $E_{3,\omega}$ can be rearranged so that $tr(\Phi_i \Phi_k^T) atan(\frac{\omega}{\lambda_i})$ is replaced by 
$tr(\Phi_i \Phi_k^H) \frac{1}{2} ( atan(\frac{\omega}{\lambda_i}) + atan(\frac{\omega}{\lambda_k^*}))$. Therefore, $E_{3,\omega}$ is similar to \eqref{eq:h2err1} and can be written as 
\begin{equation}
    E_{3,\omega} = (\mathds{1} - \alpha)^T Q_\omega (\mathds{1} - \alpha),
\end{equation}
where $Q_\omega$ is given by equation \eqref{eq:Qwmat}. Then, by using vectorisation, $E_{1,\omega}$ can be transformed as follows,
\begin{equation}
%\scriptsize
\begin{array}{rcl}
E_{1,\omega} &=& \frac{\omega}{\pi}\left (  tr(D^TD) - 2 tr(D^T D_\alpha)  + tr(D_\alpha^T D_\alpha) \right )\\
 &=&  \frac{\omega}{\pi} \left (d^T d - 2 d^T d_\alpha + d_\alpha^T d_\alpha \right ).
 \end{array}
\end{equation}
Finally, using the same vectorisation process, the last element divides into a quadratic part and a linear part,
\begin{equation}
    E_{2,\omega} = (\mathds{1} - \alpha)^T (U_{\omega} d_\alpha+ f_\omega ),
\end{equation}
where, for $i=1,\ldots,n$,
\begin{equation}
\begin{array}{rcl}
e_i^T f_\omega &=& -\frac{2}{\pi} tr(\Phi_i D^T) atan(\frac{\omega}{\lambda_i}),\\
    e_i^T U_{\omega} &=& -\frac{2}{\pi} atan(\frac{\omega}{\lambda_i}) vec(\Phi_i)^T .
    \end{array}
\end{equation}
By considering $\bar{\alpha}$ and stacking it with $d_\alpha$, the final structure of the approximation error appears,
\begin{equation}
%\scriptsize
%\begin{array}{rl}
    \Vert E \Vert_{2,\omega}^2 = \mat{cc}{\bar{\alpha}^T&d_\alpha^T}^T \mat{cc}{Q_\omega & \frac{1}{2} U_\omega\\
     \frac{1}{2} U_\omega^T& \frac{\omega}{\pi} I_{n_y n_u}}
     \mat{c}{\bar{\alpha}\\ d_\alpha}%\ldots \\
     +\mat{cc}{f_\omega^T&-2\frac{\omega}{\pi}d^T}\mat{c}{\bar{\alpha}\\ d_\alpha} + \frac{\omega}{\pi}d^Td
 %    \end{array}
\end{equation}
The constant part can be discarded and the constraints remain the same as in \cref{pb:initbin} with the additional optimisation variables $D_\alpha \in \mathbb{R}^{n_y \times n_u}$.

As $\Vert \cdot \Vert_{2,\omega}$ is a norm for rational functions, the objective is strictly convex making the overall optimisation problem a (strictly) convex mixed quadratic program. 
\end{proof}
%Due to the real variables within $D_\alpha$, the resulting optimisation problem is a Mixed Integer Quadratic Program. Since  $\Vert\cdot \Vert_{2,\omega}$ is a norm for rational transfer functions, the objective is in addition strictly convex.

\paragraph{About the initialisation}Similarly to the $\mathcal{H}_2$ case, the frequency-limited norm of the approximation error can be upper bounded as shown in \cref{prop:h2wbound}. 
Subsystems with highest individual norm may be selected initially. Note that unlike the $\mathcal{H}_2$-norm case which is parameter free, here, some subsystems may become more relevant depending on the value of the frequency bound $\omega$ as illustrated in \cref{ex:h2w}.

\begin{proposition}
\label{prop:h2wbound}
The $\mathcal{H}_{2,\omega}$-norm of the truncation error $E_\alpha$ is bounded by the sum of the discarded subsystems norms,
\begin{equation}
    \Vert E_\alpha \Vert_{2,\omega} \leq \sum_{i\in \Es} \Vert H_i \Vert_{2,\omega},
    \label{eq:h2wbound}
\end{equation}
where the individuals subsystems norms are given as 
\begin{equation}
    \Vert H_i \Vert_{2,\omega} = \frac{\Vert \Phi_i \Vert_F}{\sqrt{-2 Re(\lambda_i)}} \sqrt{-\frac{2}{\pi} Re(atan(\frac{\omega}{\lambda_i}))}. 
    \label{eq:h2wcrit_poles}
\end{equation}
\end{proposition}
\begin{proof}
The triangular inequality applied to the approximation error leads to the bound. The $\mathcal{H}_{2,\omega}$-norm of each individual system is obtained by expanding $H(-\lambda_i)$ in \eqref{eq:prh2w} and reordering the terms in the sums to pair complex conjugate components.
\end{proof}

\begin{example}
\label{ex:h2w}
Let us consider the $4$-th order model $H = H_1 + H_2$ with
\begin{equation}
\begin{array}{rcl}
    H_1(s) &=& \displaystyle \frac{2.2}{s+0.1} + \frac{1.2}{s + 0.2}\\
    &&\\
    H_2(s) &=& \displaystyle\frac{1.2}{s^2/10^4 + 0.02 s/100 + 1}
    \end{array}
\end{equation}
To reduce the system to an order $2$, there is only $\kappa(2,1,2) = 2$ solutions in the decision tree. Each coincides either with $H_1$ or with $H_2$. The $\mathcal{H}_{2,\omega}$-norm of the approximation errors are computed for varying values of the frequency bound $\omega$ ranging from $10^{-2}$ to $10^3$ and are reported in \cref{fig:h2werr} (top) together with the value of the heuristic criterion \eqref{eq:h2wcrit_poles} associated with each mode (bottom).

One can see that $H_1$ is dominant (the error is lower) for low values of $\omega$ while $H_2$ becomes dominant after $100\,rad/s$. Besides, as shown by the bottom figure in that simple case, the sorting criterion \eqref{eq:h2wcrit_poles} is coherent with the optimal result. An uncertainty area appears just before $100\,rad/s$ where the criteria for the complex eigenvalues crosses one of the real pole but not the other one.

\begin{figure}
    \centering
    \includegraphics[width=\linewidth]{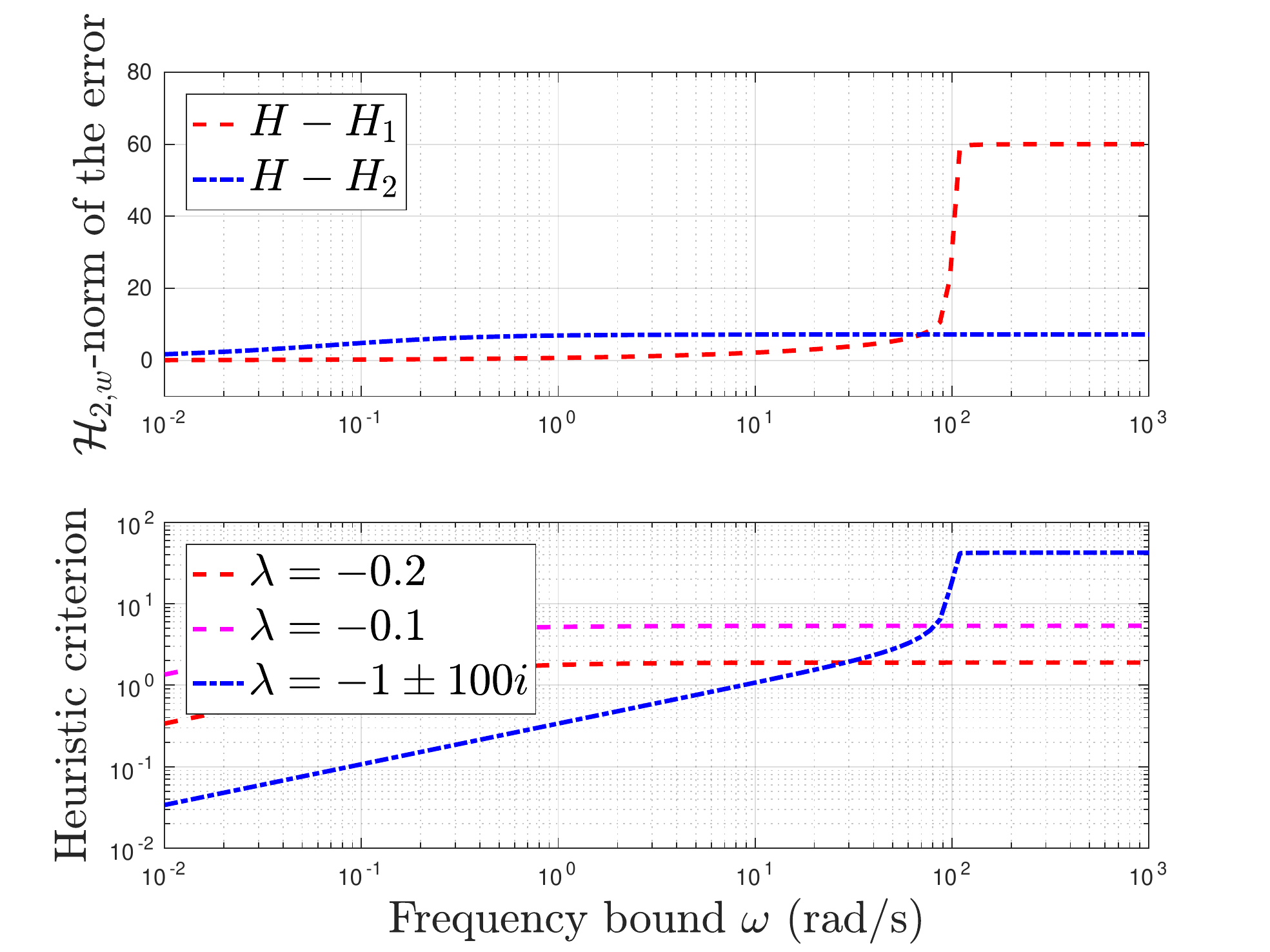}
    \caption{$\mathcal{H}_{2,\omega}$-norm of the approximation error (top) and heuristic sorting criteria associated with each eigenvalue (bottom).}
    \label{fig:h2werr}
\end{figure}

\end{example}

\subsection{In $h_{2,\tau}$-norm}
\label{ssec:h2topt}
As the $\mathcal{H}_2$ case, the direct feedthrough is here constrained to be equal to $D$ so that the resulting optimal $h_{2,\tau}$ modal truncation problem reduces to a convex binary quadratic problem as stated in \cref{thm:h2tpb}. 
\begin{theorem}[Optimal $h_{2,\tau}$ modal truncation]
\label{thm:h2tpb}
Considering the notations of \cref{pb:initbin}, the $r$-th order optimal $h_{2,\tau}$ modal truncation model $H_\alpha^\star$ is such that $D_\alpha^\star = D $ and $\alpha^\star$ is the solution of the following convex quadratic binary problem,
\begin{equation} 
\begin{array}{lc}
     \displaystyle \min_{\alpha}&   (\mathds{1} - \alpha)^T Q_\tau (\mathds{1} - \alpha)\\
     s.t.&
     \begin{array}[t]{rcl}
     \\
     \alpha&\in&\{ 0,1\}^n \\
     \mathds{1}^T \alpha&=&r\\
     M \alpha&=&0
     \end{array}
\end{array}
\label{eq:pbh2t}
\end{equation}
where $Q_\tau \in \mathbb{C}^{n\times n}$ is a hermitian matrix which entries are given, for $i,j =1,\ldots,n$, as,
\begin{equation}
     [Q_\tau]_{i,j}= (1 - e^{(\lambda_i + \lambda_j^*) \tau})  [Q]_{i,j}.
    \label{eq:Qtmat}
\end{equation}
\end{theorem}
\begin{proof}
As in the $\mathcal{H}_2$ case, the direct feedthrough $D_\alpha$ of the reduced-order model $H_\alpha$ must be equal to $D$. The error is then the same as in equation \eqref{eq:Ea} which, combined with the poles-residues expression of the $h_{2,\tau}$-norm \eqref{eq:h2tpr} leads to 
\begin{equation}
   \small
    \Vert e_\alpha \Vert_{2,\tau}^2 = \sum_{i,k \in \Is} (1- \alpha_i) \frac{tr(\Phi_i \Phi_k^T)}{\lambda_i + \lambda_k} (e^{(\lambda_i + \lambda_k) \tau}  - 1) (1 - \alpha_k),
\end{equation}
where $e_\alpha(t) = \mathcal{L}^{-1}(E_\alpha)(t)$. Again, the norm of the approximation error exhibits a quadratic structure and may be reformulated as 
\begin{equation}
   %\small
    \Vert e_\alpha \Vert_{2,\tau}^2 = (1 - \alpha)^T Q_\tau (1-\alpha),
\end{equation}
where the entries of the matrix $Q_\tau$ are given, after reordering of the elements, by \eqref{eq:Qtmat}.
\end{proof}

\paragraph{About the initialisation} In \cref{prop:h2tbound}, an upper bound of the approximation error that can motivate the selection of the initial poles-residues is presented.
\begin{proposition}
\label{prop:h2tbound}
The $h_{2,\tau}$-norm of the approximation error $\mathcal{L}^{-1}(E_\alpha) = e_\alpha$ is bounded by the sum of the discarded subsystems norms,
\begin{equation}
    \Vert e_\alpha \Vert_{2,\tau} \leq \sum_{i\in \Es} \Vert h_i \Vert_{2,\tau},
    \label{eq:h2tbound}
\end{equation}
where the individuals subsystems norms are given as 
\begin{equation}
    \Vert h_i \Vert_{2,\tau} = 
    \frac{\Vert \Phi_i\Vert_F}{\sqrt{-2 Re(\lambda_i)}}\sqrt{1 - e^{2Re(\lambda_i)\tau}}.
    \label{eq:h2tcrit_poles}
\end{equation}
\end{proposition}

\subsection{In $\mathcal{H}_\infty$-norm}
\label{ssec:hiopt}
The structure of the $\mathcal{H}_\infty$-norm makes the associated optimal modal truncation problem quite different from the previous ones. Indeed, considering \cref{pb:initbin} with the $\mathcal{H}_\infty$-norm leads to a convex, albeit non-smooth, mixed optimisation problem. While a general purpose convex optimizer may be used to solve the convex relaxation in a branch and bound process, the specific structure of the $\mathcal{H}_\infty$-norm can be exploited to derive an alternative equivalent problem which may be more direct to address. The latter in detailed in \cref{thm:hinf}.

\begin{theorem}[Optimal $\mathcal{H}_\infty$ modal truncation] 
\label{thm:hinf}
Let us consider the complex diagonal form of the approximation error $E_\alpha$,
\begin{equation}
    A_e = \Delta,\, B_e = B_\Delta,\, C_e = C_\Delta \Delta_\alpha,\, D_e = D - D_\alpha,
    \label{eq:errrea}
\end{equation}
with $\Delta_\alpha = I_n - diag(\alpha_1,\ldots,\alpha_n)$. Then the $r$-th order optimal $\mathcal{H}_\infty$ modal truncation model $H_\alpha^\star$ is given by the solution of the following mixed SDP,
\begin{equation}
    \begin{array}{cc}
    \displaystyle \min_{\gamma, P, \alpha, D_\alpha} &\displaystyle \gamma \\
    s.t.&
    \begin{array}[t]{rcl}
        &&\\
        \alpha& \in & \{0,1\}^n \\
        D_\alpha&\in&\mathbb{R}^{n_y\times n_u}\\
        \gamma &\in &\mathbb{R}_+\\
        P &\in &\mathbb{S}^{n}_{++}\\
        \mathds{1}^T \alpha &=&r\\
        M \alpha &=&0 \\
    M_e(\alpha) &\prec& 0
    \end{array}
    \end{array}
    \label{eq:pbinf3}
\end{equation}
where the last constraint is a Linear Matrix Inequality (LMI) characterised by the matrix
\begin{equation}
M_e(\alpha) = 
\left [
    \begin{array}{ccc}
    A_e^T P + P A_e & P B_e & C_e^T\\
    \star& - \gamma I & D_e^T \\ 
    \star& \star& -\gamma I 
    \end{array}
\right ].
\label{eq:lastLMI}
\end{equation}
\end{theorem}

\begin{proof}
As its objective is convex, by adding a slack variable $\gamma \geq 0$, the $\mathcal{H}_\infty$ modal approximation problem can be re-written equivalently (see \cite{boyd:2004:convex}) as 
\begin{equation}
    \begin{array}{cc}
    \displaystyle \min_{\gamma, \alpha, D_\alpha} &\displaystyle \gamma \\
    s.t.&
    \begin{array}[t]{rcl}
        &&\\
        \alpha& \in & \{0,1\}^n \\
        D_\alpha&\in&\mathbb{R}^{n_y\times n_u}\\
        \gamma &\in &\mathbb{R}_+\\
        \mathds{1}^T \alpha &=&r\\
        M \alpha &=&0 \\
        \left \Vert E_\alpha \right \Vert_\infty &\leq &\gamma\\
    \end{array}
    \end{array}
    \label{eq:pbinf2}
\end{equation}
The last constraint can then be transformed even further.  Indeed, let $(\Delta, B_\Delta, C_\Delta, D)$ denote the complex realisation associated with the DCF of $H$ so that $H(s) = C_\Delta(sI_n - \Delta)^{-1} B_\Delta + D$. Equation \eqref{eq:errrea} then represents a complex realisation associated with the approximation error $E_\alpha = H - H_\alpha$.

By using the Bounded Real Lemma (see e.g. \cite{Scherer:1997}), the constraints $\Vert E_\alpha \Vert_\infty \leq \gamma$ can be traded by considering additional slack variables as a symmetric and positive definite matrix $P \in \mathbb{S}^{n}_{++}$ such that the following matrix inequality is satisfied,
\begin{equation}
\left [
    \begin{array}{cc}
    A_e^T P + PA_e + C_e^T C_e  & P B_e + C_e^T D_e \\
    \star               & D_e^T D_e - \gamma^2 I_{n_u}
    \end{array}
\right ] \prec 0.
\end{equation}
The inequality can further be transformed by a Schur complement leading to the LMI $M_e(\alpha) \prec 0$ where $M_e (\alpha) $ given by equation \eqref{eq:lastLMI}.
\end{proof}

The convex relaxation of the binary constraints in \eqref{eq:pbinf3} leads to a SDP which can be easily formulated using a modelling framework such as YALMIP \cite{Lofberg2004} and solved by associated solvers such as SeDuMi \cite{polik2007sedumi}. Still, while a SDP is convex, it remains difficult to solve, especially as the dimension of the problem increases. For these reasons, the optimal $\mathcal{H}_\infty$ modal truncation is restricted to models of moderate dimension $n$. 

The LMI matrix \eqref{eq:lastLMI} is complex-valued. Should the SDP solver only handle real matrices, then an equivalent real-valued matrix should be considered as detailed in  \cref{prop:eqreal}.

\begin{proposition}
\label{prop:eqreal}
Let us consider the set of indexes associated with real poles $\mathcal{I}^R= \{ i \in \Is, \lambda_i \in \mathbb{R}\}$ and the set $\mathcal{I}^C$ defined in \eqref{eq:Icdef}. Let us define the unitary transformation matrix $T \in \mathbb{C}^{n\times n}$ such that
\begin{equation}
    T_{i,i} =1,\,i\in \mathcal{I}^R,
\end{equation}
and such that for $\{i,j\} \in \mathcal{I}^C$,
\begin{equation}
\begin{array}{ccccc}
    T_{i,i} &=& T_{j,i} &=&  1/\sqrt{2},  \\
    T_{i,j} &=& -T_{j,j} &=&  \imath / \sqrt{2}.
\end{array}
\end{equation}
Then, assuming that the constraint $M \alpha = 0$ is satisfied, the complex realisation \eqref{eq:errrea} can be replaced by the following real realisation,
\begin{equation}
    \tilde{A}_e = T^{H} A_e T,\, \tilde{B}_e = T^{H}B_e,\, \tilde{C}_e = C_\Delta T \Delta_\alpha.
\end{equation}
\end{proposition}

\begin{proof}
The real matrices $ \tilde{A}_e$, $ \tilde{B}_e$ and $C_\Delta T$ correspond to the standard real-valued realisation associated with the DCF obtained by combining complex conjugate elements. The diagonal matrix $\Delta_\alpha$ acts as a filter on the output matrix which may be applied before or after the transformation to real form provided that the complex pairs are kept together.

Indeed, as long as this is the case, i.e. that the constraint \eqref{eq:cst_realness} is satisfied, the matrices $\Delta_\alpha$ and $T$ commutes. More specifically, as $\Delta_\alpha$ is idempotent, $\Delta_\alpha T = T \Delta_\alpha$ is equivalent to $\Delta_\alpha T (I - \Delta_\alpha) = 0$. Looking at each entry leads to $ \alpha_i T_{i,j} (1 - \alpha_j) = 0$ for $i,j \in \Is$ which is satisfied considering the zero elements of $T$ and the dependencies between the $\alpha_i$ for complex poles. Therefore, $\tilde{C}_e = C_\Delta \Delta_\alpha T =  C_\Delta T \Delta_\alpha$ which leads to the result.
\end{proof}

\section{Numerical illustrations}
\label{sec:app}

In \cref{ssec:acad}, an academic example is used to highlight the different nature of poles dominance depending on the considered norm. Then, a more realistic example is considered in \cref{ssec:domid} for $\mathcal{H}_2$ dominant modes identification.

\subsection{Academic example}
\label{ssec:acad}
\begin{figure}
    \centering
    \includegraphics[width=\linewidth]{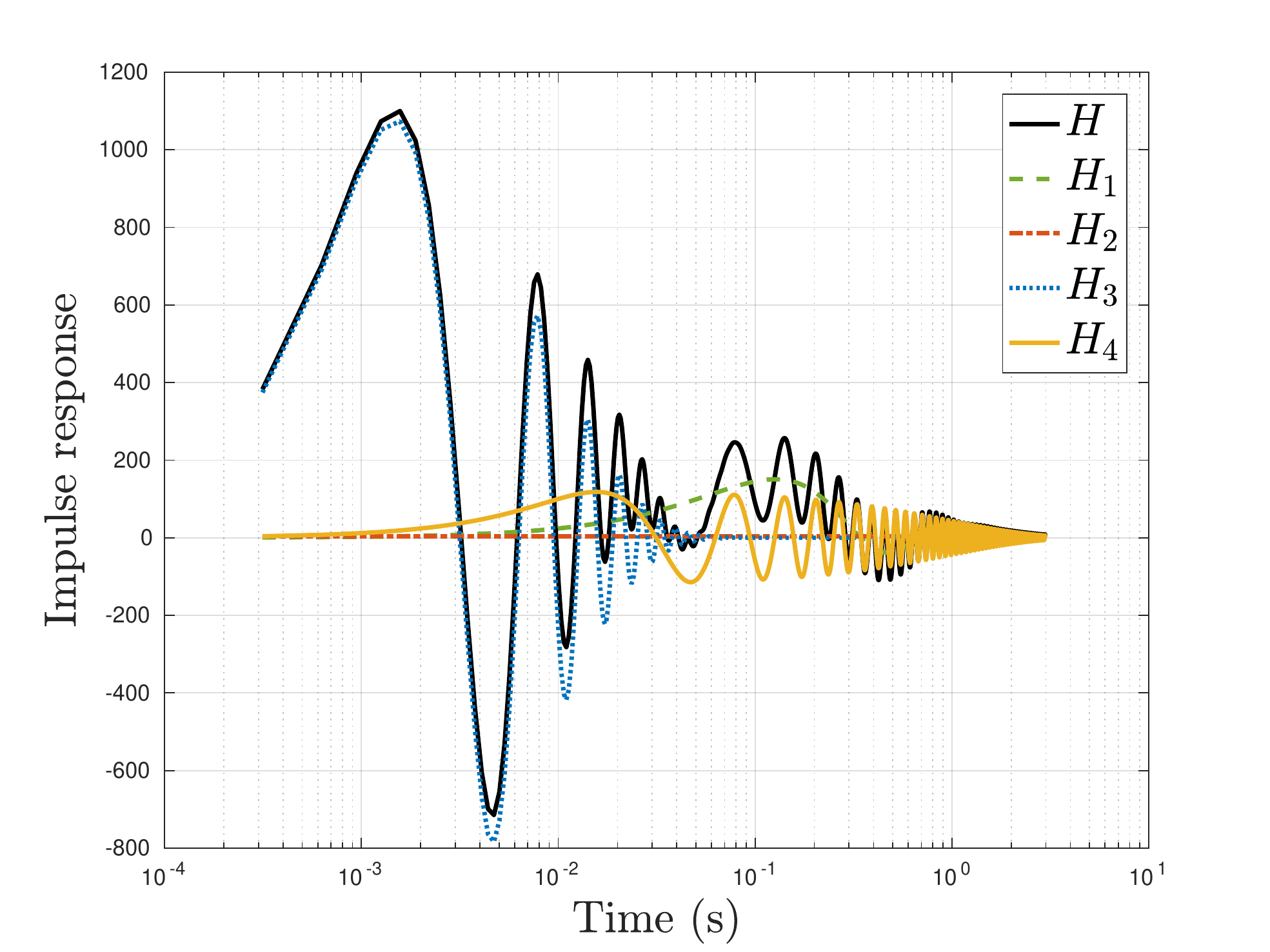}
    \caption{Impulse responses of $H$ and its components with the time axis in log scale.}
    \label{fig:aca_impulse}
\end{figure}

To highlight the differences between the different norms within the context of modal truncation, let us consider the following $8$-th order academic model,
\begin{equation}
    H(s) = H_1(s) + H_2(s) + H_3(s) + H_4(s),
\end{equation}
where the subsystems are given as follows,
\begin{equation}
\begin{array}{rcl}
    H_1(s) &=& \displaystyle \frac{25}{s^2/10^2 + 0.8 s/10 + 1 }, \\
    &&\\
    H_2(s) &=& \displaystyle \frac{2.2}{s+0.1} + \frac{1.2}{s + 0.2},\\
    &&\\
    H_3(s) &=&\displaystyle\frac{1.25}{s^2/10^6 + 0.2s/10^3 s + 1}, \\&&\\
    H_4(s) &=& \displaystyle\frac{1.2}{s^2/10^4 + 0.02 s/100 + 1}.
    \end{array}
\end{equation}
The impulse responses of $H$ and its components are plotted in \cref{fig:aca_impulse}. The gains of their frequency responses are plotted in \cref{fig:aca_bode}. To reduce this model to an order $r=2$ by modal truncation, only $\kappa(2,3,2) = 4$ combinations are possibles. These combinations are actually associated with each subsystem $H_i$.

\begin{figure}
    \centering
    \includegraphics[width=\linewidth]{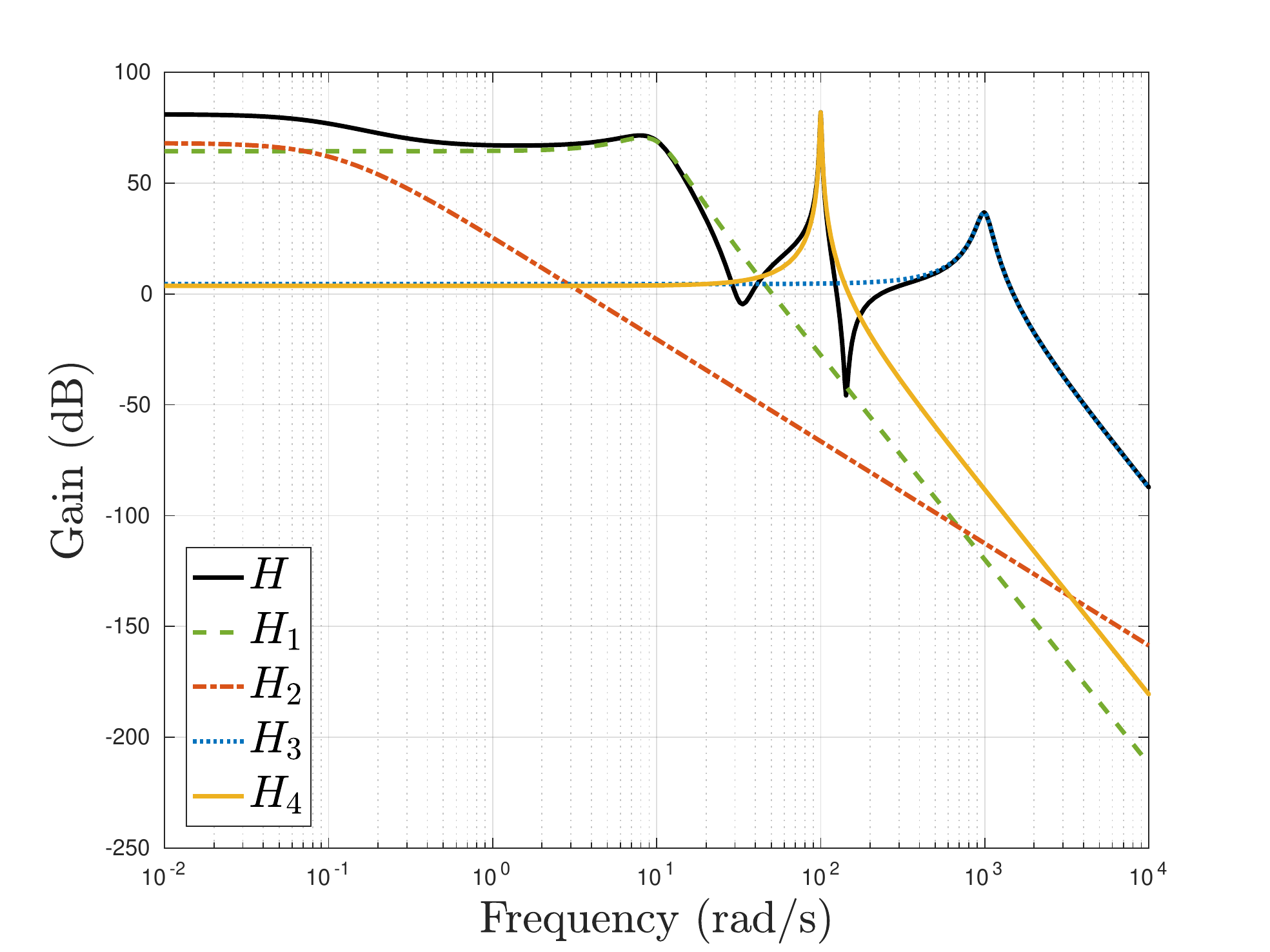}
    \caption{Gain of the frequency responses of $H$ and its components.}
    \label{fig:aca_bode}
\end{figure}

The different norms of the approximation errors are computed and reported in \cref{tab:aca_err} with $\omega = 0.1$ and $T=0.2$. Note that no direct feedthrough is considered here.

In this simple case, the results could, in a sense, be inferred from the gain diagram in \cref{fig:aca_bode}:
\begin{itemize}
\item $H_1$ has visually an important mean contribution and should therefore play an important role in the $\mathcal{H}_2$-norm.
\item $H_2$ has a large contribution only in low frequency below $0.1$rad/s and should therefore be important w.r.t. the $\mathcal{H}_{2,\omega}$-norm. 
\item $H_3$ contains the fastest modes and should therefore be dominant at the beginning of the impulse response thus dominating the $h_{2,T}$-norm when $T =0.2$. This is validated by the impulse responses in \cref{fig:aca_impulse}.
\item $H_4$ has the highest gain and should therefore be dominant w.r.t the $\mathcal{H}_\infty$-norm.
\end{itemize}

However, as illustrated in the next example, such a heuristic analysis is no longer tractable on realistic models and the systematic approach detailed in this work then shows its benefits.

\begin{table}
    \centering
    \begin{tabular}{c|cccc}
    Subsystem\textbackslash Norm& $\mathcal{H}_2$&$\mathcal{H}_{2,\omega}$&$h_{2,T}$&$\mathcal{H}_\infty$\\ \hline
    $H_1$& $\mathbf{87.07}$ & $5.21$            & $71.40$ & $60.07$ \\
    $H_2$& $106.89$         & $\mathbf{4.90}$   & $89.37$&$60.05$ \\
    $H_3$& $88.04$          & $9.26$            & $\mathbf{65.16}$&$60.06$ \\
    $H_4$& $89.74$          & $9.27$            & $84.22$&$\mathbf{56.25}$ \\
    \end{tabular}
    \caption{Norms of the approximation errors $H - H_i$, $i=1,\ldots,4$}
    \label{tab:aca_err}
\end{table}

\subsection{Application to dominant modes selection}
\label{ssec:domid}

\begin{figure}
    \centering
    \includegraphics[width=\linewidth]{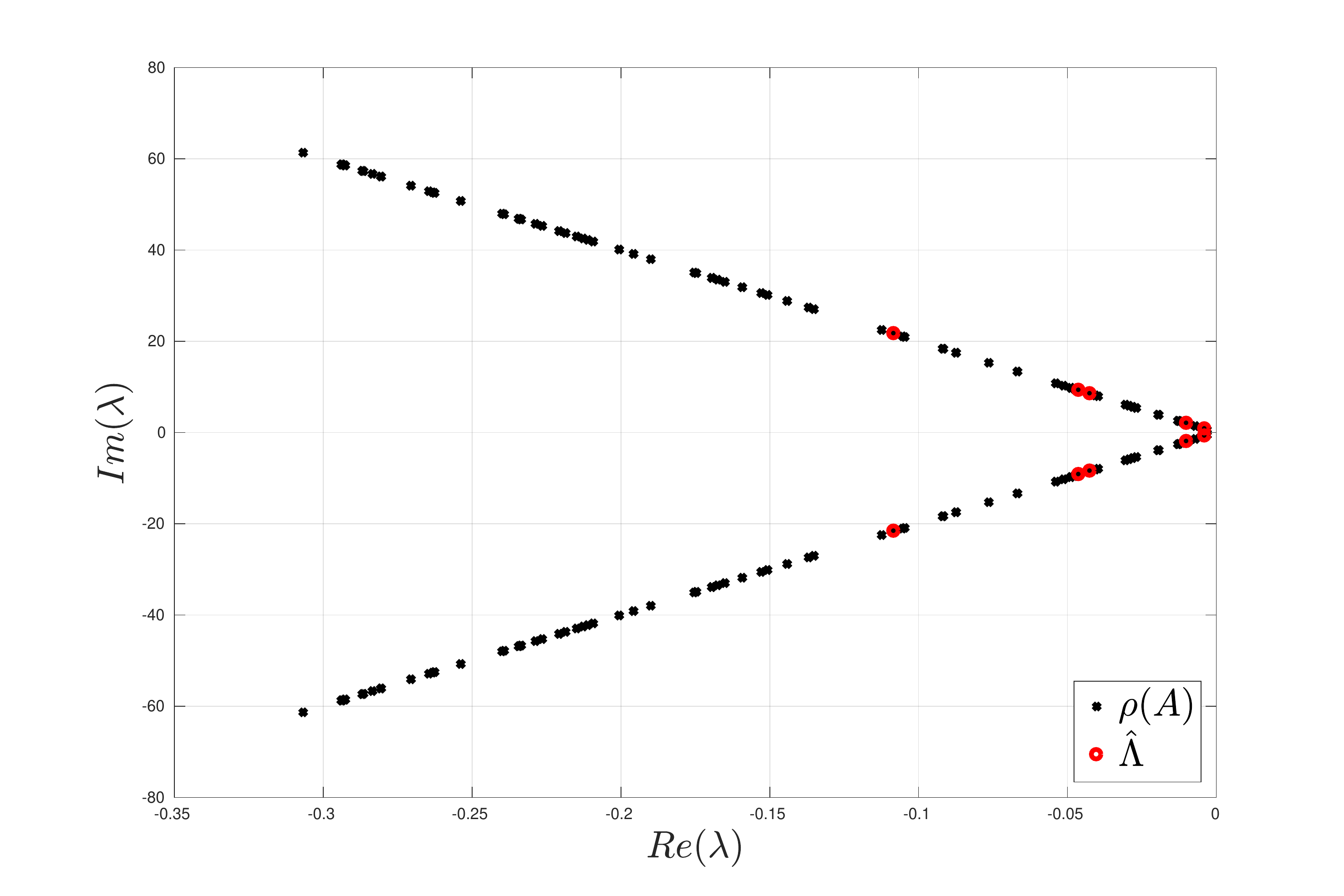}
    \caption{Initial set of poles $\rho(A)$ and subset $\hat{\Lambda}$ containing the $10$ $\mathcal{H}_2$ dominant poles.}
    \label{fig:serious2}
\end{figure}

In this example, let us consider the \texttt{ISS2} model from \cite{complib}. It is a $270$-th order model with $3$ inputs and outputs. It has only complex eigenvalues and the number of combinations for modal approximation grows as $\kappa(0,135,r)$ where $r$ should be chosen even. Assuming that we are looking for the $r=10$ dominant poles, then there are more than $10^{8}$ possible combinations.

The $\mathcal{H}_2$ optimal modal truncation problem is solved with a simple branch and bound algorithm starting from the initialisation scheme suggested in Section \ref{ssec:h2opt}. The set of initial poles $\rho(A)$ are plotted together with the $r$ $\mathcal{H}_2$ dominant ones in \cref{fig:serious2}. The singular values associated with the initial and reduced-order transfer functions are plotted in \cref{fig:serious1}.

The initialisation heuristic choice turns out to be almost optimal in that case as only $3$ nodes lead to an improvement of the $\mathcal{H}_2$ error. This is not surprising since the model represents a highly flexible structure which dominant modes are, in a sense, easily distinguishable due to the large magnitude of some of the associated residues. Indeed, as shown by \cref{fig:serious1}, dominant modes are mainly associated with peaks in the frequency-domain response. However, this may not be as clear in general, especially when the model has some well damped dynamics. 

It is clear from \cref{fig:serious2} that dominant poles are not necessarily the ones with lowest natural frequency and that the associated residues play and important role. Indeed, the dominance must be understood in and input-output sense through the lens of the actuators, sensors and eigenvectors. The pole dominance as defined here is therefore associated with and input-output setting.

\begin{figure}
    \centering
    \includegraphics[width=\linewidth]{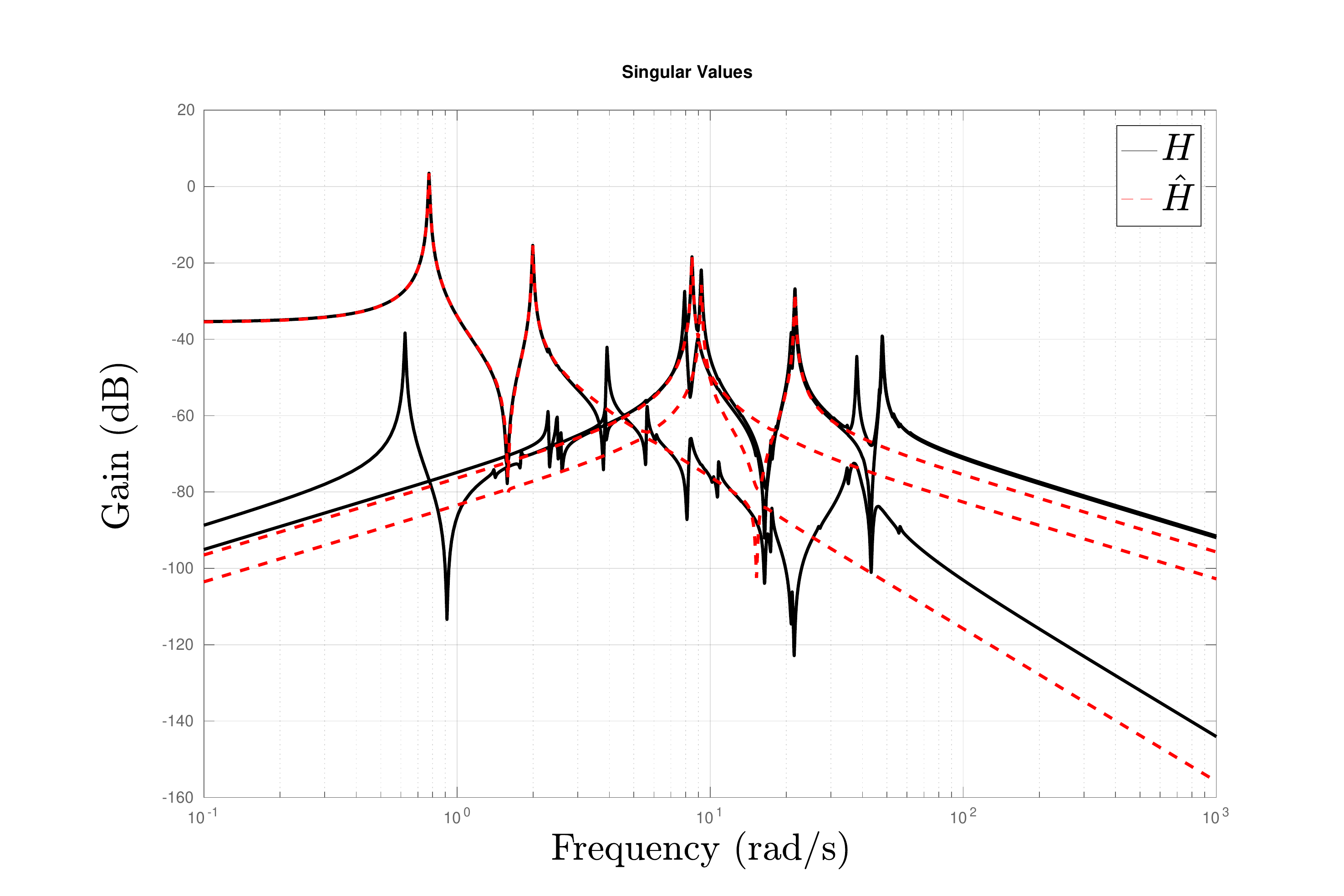}
    \caption{Singular values of the initial model $H$ and its $10$-th order $\mathcal{H}_2$ optimal modal truncation $\hat{H}$.}
    \label{fig:serious1}
\end{figure}

\section{Conclusion}
\label{sec:ccl}

This article revisits the well-known modal truncation technique from an optimisation point of view. In particular, dominant poles are defined as the solution of the associated optimal modal truncation problem  with respect to different systems norms. The latter reduces to a convex integer or mixed integer program depending on the considered norm: $\mathcal{H}_2$-norm and its frequency/time limited variants, $\mathcal{H}_\infty$-norm. 

The same approach may be developed for other norms or quantities of interest. For instance, one may be interested in knowing which poles are dominant in response to some specific input other than an impulse such as a step, a combination of sine, etc. This can give meaningful insights to determine the main dynamics in an automatic way for control design purposes.

In this article, the choice has been made to stick to the original modal truncation framework, i.e. the poles and the associated residues are kept fixed (excepted the term $D$). Therefore, the resulting model is only optimal among all the models with same residues and poles. To improve its matching with the initial large-scale model, the residues (one side in the MIMO case) should be considered as free variables. To avoid the multiplication of the binary variables with the residues in the objective, the optimisation problem may then be modified with the big-M technique such that each binary variable behaves as an activation variable through the constraint $\Vert \Phi_i \Vert_F \leq  \alpha_i M$ where $M$ is large enough. The problems then remain convex and may be solved similarly.

More generally, the underlying concept of this article consists in determining the dominant components of the additive decomposition of a LTI model. While the diagonal canonical form is considered here, the idea may be generalised to other decompositions. An interesting candidate is the block diagonal form (see e.g. the discussion in \cite{moler:1978:nineteen}). Indeed, the latter preserve the Jordan blocks and can thus be safely applied to systems with multiple eigenvalues. However this induces changes for the computation of the norms that must be integrated to the associated optimisation problems.

%\section{Funding}
%Part of this research was carried out at the Jet Propulsion Laboratory, California Institute of Technology, under a contract with the National Aeronautics and Space Administration (80NM0018D0004).

\bibliographystyle{plain}
\bibliography{biblio}

\end{document}